\documentclass[12pt, leqno]{article}

\usepackage{amsmath,amsthm}
\usepackage{amssymb}

\numberwithin{equation}{section}

\newtheorem{theorem}{\noindent Theorem}[section]

\newtheorem{lemma}{\noindent Lemma}[section]
\newtheorem{corollary}{\noindent Corollary}[section]

\def\N{{\mathbb N}}
\def\Z{{\mathbb Z}}
\def\R{{\mathbb R}}

\def\A{{\mathcal A}}
\def\S{{\mathcal S}}

\begin{document}

\baselineskip=17pt

\title{THE NUMBER OF INTEGER POINTS CLOSE TO A POLYNOMIAL}

\author{Patrick Letendre}

\date{}

\maketitle

\renewcommand{\thefootnote}{}

\footnote{2010 \emph{Mathematics Subject Classification}: 11J54.}

\footnote{\emph{Key words and phrases}: fractional parts, polynomials.}

\renewcommand{\thefootnote}{\arabic{footnote}}
\setcounter{footnote}{0}

\vspace{-2cm}

\begin{abstract}
Let $f(x)$ be a polynomial of degree $n \ge 1$ with real coefficients and let $X \ge 2$ and $\delta \ge 0$ be real numbers. Let $\|\cdot\|$ be the distance to the nearest integer. We obtain upper bounds for the number of solutions to the inequality $\|f(x)\| \le \delta$ with $x \in [X,2X] \cap \N$.
\end{abstract}

\section{Introduction}

Consider a polynomial with real coefficients
\begin{equation}\label{f}
f(x) := \alpha_n x^n + \dots + \alpha_1 x + \alpha_0
\end{equation}
of degree $n \ge 1$. Let $\delta, X \in \R$ be such that $0 \le \delta \le 1/4$ and $X \ge 2$. Following the notation of the papers of Huxley and Sargos \cite{mnh:ps} and \cite{mnh:ps:2}, we put
\begin{equation}\label{gamma}
\Gamma_{\delta} := \{(x,y) \in \R:\ x \in [X,2X], |y-f(x)| \le \delta\}
\end{equation}
and
$$
\S := \#(\Gamma_{\delta} \cap \Z^2).
$$
We are interested in estimating $\S$.

In the case where $\delta=0$, the problem is to estimate the number of integer solutions to the equation $f(x)=y$ with $x \in [X,2X]$. We suppose at first that at least one of the coefficients is irrational. We then show that there are at most $n$ solutions. Indeed, if we can find $n+1$ solutions then, by using the fact that the Vandermonde determinant is nonzero, we can solve the system
$$
\begin{pmatrix} 
y_1 \\
\vdots \\
y_{n+1}
\end{pmatrix} = \begin{pmatrix} 
1 & x_1 & \cdots & x_1^n \\
\vdots & \vdots & \cdots & \vdots \\
1 & x_{n+1} & \cdots & x_{n+1}^n 
\end{pmatrix}\begin{pmatrix} 
\alpha_0 \\
\vdots \\
\alpha_{n}
\end{pmatrix}
$$
and thus recover the coefficients of $f(x)$. Since these are rational, we get a contradiction. Moreover this upper bound is optimal since the equation
$$
y = \varkappa x (x-1) \cdots (x-n+1)
$$
has the $n$ solutions $(0,0)$,\dots, $(n-1,0)$ no matter the value of $\varkappa \neq 0$.

Now, in the case where all the coefficients are rational, then the problem is very different and there are infinitely many solutions with $x \in \Z$ in general. The papers \cite{svk} and \cite{svk:ts} provide a very satisfactory study of this question.

In our situation, the case $\delta>0$ is more interesting.  We first establish in Theorem 1.1 that the number of solutions is essentially less than a quantity that we will assume to be small plus the contribution arising from a single polynomial of degree at most $n$ with rational coefficients. We therefore need to solve the problem of the previous paragraph but with an unknown polynomial.

\begin{theorem}
Let $f(x)$ be as in \eqref{f}. Then,
$$
\S \ll_n \delta^{\frac{2}{n(n+1)}}X+\mathcal{R}
$$
where $\mathcal{R}$ is the maximal number of integer points in $\Gamma_{\delta}$ that are all on a polynomial of degree at most $n$.
\end{theorem}

\begin{corollary}
Let $f(x)$ be as in \eqref{f}. Assume that the inequality
$$
\left|\alpha_n-\frac{r}{s}\right| \le \frac{1}{s^2}
$$
holds for some integers $r \in \Z$ and $s \in [1,X^n]$ with $gcd(r,s)=1$. Then,
$$
\S \ll_{n,\epsilon} \delta^{\frac{2}{n(n+1)}}X+\frac{X}{s^{1/n}}+X^{\epsilon}
$$
for each $\epsilon > 0$. For $n = 1$ the third term can be replaced by 1.
\end{corollary}

\begin{corollary}
Let $f(x)$ be as in \eqref{f}. Assume that the inequality
\begin{equation}\label{hypc2}
\left|\alpha_n-\frac{r}{s}\right| \ge \frac{c_1}{s^{\frac{n+3}{2}}}
\end{equation}
holds for all $r \in \Z$ and $s \in \N$, for some constant $c_1 > 0$. Then,
$$
\S \ll_{n,\epsilon,c_1} \delta^{\frac{2}{n(n+1)}}X+X^{\epsilon}
$$
for each $\epsilon > 0$. For $n = 1$ the second term can be replaced by 1.
\end{corollary}

A point $(x,y)$ of $\Gamma_{\delta} \cap \Z^2$ is often noted $M$. The function $\omega(q)$ counts the number of distinct prime divisors of $q$. The notation $A \ll B$ and $B \gg A$ mean that the estimate $|A| \le cB$ holds for some constant $c > 0$. In what follows, $\epsilon>0$ is a real number taken arbitrarily small and may differ at each occurrence.

\section{Definitions and preliminary lemmas}

We order the set of points $(x,y) \in \Gamma_{\delta} \cap \Z^2$ according to their first coordinate $x$. Two or more points are said to be {\it consecutive} in $\Gamma_{\delta} \cap \Z^2$ if they follow each other in this sequence.

Consider the polynomial of degree at most $n$
\begin{equation}\label{P}
P(x) := \frac{a_n}{q_n} x^n + \dots + \frac{a_1}{q_1} x + \frac{a_0}{q_0}
\end{equation}
where each coefficient is rational and where $gcd(a_i,q_i)=1$ and $q_i \ge 1$ for $i=0,\dots,n$. We then let $q:=lcm(q_n,\dots,q_1,q_0)$ be the smallest positive integer for which $qP(x) \in \Z[x]$. We assume that $deg\ P=d$.

Consider the curve
\begin{equation}\label{gam}
\gamma:=\{(x,y):\ x \in \R, y = P(x)\}.
\end{equation}

A {\it major arc} $\A$ of {\it equation} $y=P(x)$ is a set of at least $n+2$ consecutive points $(x,y)$ in $\Gamma_{\delta} \cap \Z^2$ that satisfy the equation $y = P(x)$. A {\it proper major arc} is a major arc which is also a subset of a connected component of $\gamma \cap \Gamma_{\delta}$. Moreover, $q$ is said to be the {\it denominator} of $\A$. The {\it length} of $\A$, composed of the consecutive points $M_1, \dots, M_J$, is $x_J-x_1$.

\begin{lemma}\label{lon}
Let $T(x):=\beta_n x^n + \dots + \beta_1 x + \beta_0 \in \R[x]$ be a polynomial of degree $n \ge 1$ and let $\Delta \ge 0$ be a real number. Let also $I$ be an interval of length $L$ for which the inequality
$$
|T(x)| \le \Delta \qquad (x \in I)
$$
holds. Then,
$$
L \le 2e\left(\frac{\Delta}{|\beta_n|}\right)^{1/n}.
$$
This last inequality is strict if $\Delta>0$.
\end{lemma}

\begin{proof}
If $\Delta = 0$ the result follows from the fundamental theorem of algebra. For $\Delta > 0$, the proof is similar to that of Lemma 4 from \cite{mnh:ps} in which we use the elementary inequality $n! > \frac{n^n}{e^n}$ for $n \ge 1$.
\end{proof}

\begin{lemma}\label{mult}
Let $T(x):=\beta_n x^n + \dots + \beta_1 x + \beta_0 \in \R[x]$ be a polynomial of degree $n \ge 1$ and let $\Delta \ge 0$ be a real number. The number of connected components of
$$
\{(x,y) \in \R^2 :\ y=T(x)\} \cap \{(x,y) \in \R^2 :\ |y| \le \Delta\}
$$
is at most $n$.
\end{lemma}

\begin{proof}
If $\Delta = 0$ then the result follows from the fundamental theorem of algebra. So, let us assume that $\Delta > 0$ and set
$$
G(x):=(T(x)-\Delta)(T(x)+\Delta)= \beta^2_n\prod_{\substack{\zeta \\ G(\zeta)=0}}(x-\zeta)^{n_\zeta}
$$
where the $\zeta$'s are pairwise distinct. Each component that is not reduced to a single point has an entry point and an exit point. Each of these two points corresponds to a root of $G(x)$ of odd multiplicity. The other components correspond to a root of $G(x)$ of even multiplicity. We deduce then that each connected component corresponds to a factor of degree at least 2 of the polynomial $G(x)$. The result follows.
\end{proof}

\begin{lemma}\label{minor}
Let $M_1, \dots, M_{n+2} \in \Gamma_{\delta} \cap \Z^2$ be ordered points according to $x_1 < \dots < x_{n+2}$. Set
\begin{equation}\label{lam-def}
\Lambda(M_1,\dots,M_{n+2}):=\begin{vmatrix} 1 & x_1 & x_1^2 & \cdots & x_1^n & y_1\\ \vdots & \vdots & \vdots & \cdots & \vdots & \vdots \\ 1 & x_{n+2} & x_{n+2}^2 & \cdots & x_{n+2}^n & y_{n+2} \end{vmatrix}.
\end{equation}
Then, there are two possibilities:
\begin{itemize}
\item[\bf(i)] $\Lambda(M_1,\dots,M_{n+2}) \neq 0$ in which case $|x_{n+2}-x_1| \ge \left(\frac{1}{(n+2)\delta}\right)^{\frac{2}{n(n+1)}}$,
\item[\bf(ii)] $\Lambda(M_1,\dots,M_{n+2}) = 0$ in which case all the points are on a curve \eqref{gam}.
\end{itemize}
\end{lemma}

\begin{proof}
We start by assuming that $|\Lambda(M_1,\dots,M_{n+2})| \ge 1$ and write
$$
y_i = f(x_i) + \delta_i \qquad (i=1, \dots, n+2).
$$
We then introduce these relations in \eqref{lam-def} and simplify the determinant which leads to the identity
$$
\Lambda(M_1,\dots,M_{n+2}) = \begin{vmatrix} 1 & x_1 & x_1^2 & \cdots & x_1^n & \delta_1\\ \vdots & \vdots & \vdots & \cdots & \vdots & \vdots \\ 1 & x_{n+2} & x_{n+2}^2 & \cdots & x_{n+2}^n & \delta_{n+2} \end{vmatrix}.
$$
Expanding according to the last column and using the well-known formula to evaluate a Vandermonde determinant, we obtain that
$$
|\Lambda(M_1,\dots,M_{n+2})| \le (n+2)\delta |x_{n+2}-x_1|^{\frac{n(n+1)}{2}}.
$$
The inequality announced in {\bf (i)} then follows.

We now move to the second case. Since $\Lambda(M_1,\dots,M_{n+2}) = 0$, we have a linear dependence between the columns and we deduce that the $n+2$ points satisfy a polynomial equation $y = b_n x^n + \dots + b_1 x + b_0$ where each coefficient $b_i$ ($i=0,\dots,n$) is rational, since we can express them as a quotient of subdeterminants of $\Lambda(M_1,\dots,M_{n+2})$. This proves {\bf (ii)}.
\end{proof}

\begin{lemma}\label{qsol}
Consider the polynomial $P(x)$ defined in \eqref{P}. The number $W$ of solutions in $x$ belonging to an interval of length $L$ to the equation
$$
P(x) \equiv 0 \pmod{1}
$$
satisfies the inequality
$$
W \ll n\frac{L}{q^{1/n}}+n^{\omega(q)}.
$$
\end{lemma}

\begin{proof}
For $d=0$ the result is trivial. For $d > 0$, it follows from Theorem 2 of \cite{svk}, Theorem 1.1 of \cite{svk:ts} and the fact that $d \le n$.
\end{proof}

\begin{lemma}\label{qdensity}
Consider the polynomial $P(x)$ defined in \eqref{P} and assume that $n \ge 2$. Let $W$ be the number of solutions in $x$ belonging to an interval of length $L$ to the equation
\begin{equation}\label{qsys}
P(x) \equiv 0 \pmod{1}.
\end{equation}
Then, for every integer $k \ge n+1$, the inequality
$$
W < k\left(\frac{L}{q^{\frac{1}{n}-\frac{n-1}{n(k-1)}}}+1\right)
$$
holds.
\end{lemma}

\begin{proof}
We write the factorization of $q$ into pairwise distinct primes as $p_1^{\lambda_1} \cdots$ $p_t^{\lambda_t}$. It is established in the proof of Lemma 1 of \cite{svk} that, for each fixed $1 \le m \le t$, there exists a polynomial $U_{m}(x)=\prod_{i=1}^{d}(x-\eta_{m,i})$ of degree $d$, where $\eta_{m,i} \in \Z$ for each $i=1,\dots,d$, whose zeros contain those of $qP(x)$ modulo $p_m^{\lambda_m}$. By the Chinese remainder theorem, we can solve the system
\begin{eqnarray*}
\eta_{i} \equiv \eta_{1,i} && \pmod{p_1^{\lambda_1}} \\
& \vdots &\\
\eta_{i} \equiv \eta_{t,i} && \pmod{p_t^{\lambda_t}}
\end{eqnarray*}
for $\eta_i$ for each $1 \le i \le d$. Thus we can construct the polynomial $U(x)=x^{n-d}\prod_{i=1}^{d}(x-\eta_{i})$ of degree $n$ whose zeros contain those of $qP(x)$ modulo $q$. We can now proceed with the argument. We first assume that $L \le q^{\frac{1}{n}-\frac{n-1}{n(k-1)}}$. Then, we assume that we can find $k$ solutions $x_1 < \dots < x_k$ to equation \eqref{qsys}. We write $k=:ln+w$ with $l \ge 1$, $1 \le w \le n$ and consider the determinant
$$
D:=\begin{vmatrix} 1 & x_1 & \cdots & x_1^{n-1} & U(x_1) & x_1U(x_1) & \cdots  & x_1^{w-1}U(x_1)^l \\
\vdots & \vdots & \cdots & \vdots & \vdots & \vdots & \cdots  & \vdots \\
1 & x_k & \cdots & x_k^{n-1} & U(x_k) & x_kU(x_k) & \cdots  & x_k^{w-1}U(x_k)^l \\
\end{vmatrix}.
$$
On the one hand, $D$ is equal to a Vandermonde determinant, so that in particular
\begin{equation}\label{qd1}
1 \le D = \prod_{1 \le i < j \le k}(x_j-x_i) < L^{\binom{k}{2}}.
\end{equation}
On the other hand,
\begin{equation}\label{qd2}
q^{\frac{k^2}{2n}-\frac{k}{2}} \mid D.
\end{equation}
By comparing \eqref{qd1} with \eqref{qd2} we get a contradiction. This proves that an interval of length $q^{\frac{1}{n}-\frac{n-1}{n(k-1)}}$ contains less than $k$ solutions. The general result follows by splitting the interval into at most $\frac{L}{q^{\frac{1}{n}-\frac{n-1}{n(k-1)}}}+1$ subintervals of length $q^{\frac{1}{n}-\frac{n-1}{n(k-1)}}$.
\end{proof}

\begin{lemma}\label{rep-loc}
Let $h \in C^{n}([X,2X])$ and $\Delta >0$ be a real number. Assume that the inequality $|h(x)| \le \Delta$ holds for all $x$ contained in an interval $I \subseteq [X,2X]$ of length $L$. Let $x_0$ be such that $|h(x_0)| \ge \sigma$ for some real number $\sigma > \Delta$. Let $K$ be the distance from $x_0$ to the furthest point in $I$. Then, for $n \ge 2$, we have
\begin{equation}\label{rep-arc}
K \ge \min\left(\left(\frac{\sigma n!}{2(2n)^n}\right)^{1/(n-1)}\frac{L}{\Delta^{1/(n-1)}},\left(\frac{\sigma n!}{2}\right)^{1/n}\frac{1}{|h^{(n)}(\xi)|^{1/n}}\right)
\end{equation}
for some $\xi \in I$. For $n=1$ we have
\begin{equation}\label{rep-arc1}
K \ge \frac{\sigma-\Delta}{|h^{(1)}(\xi)|}
\end{equation}
for some $\xi \in I$.
\end{lemma}

\begin{proof}
This is a restatement/generalization of the last part of Lemma $19$ of \cite{mnh}. Starting with the case $n \ge 2$, we will assume that $x_0 < x$ whenever $x \in I$, the other case being similar. Let $x^\prime$ be the point in $I$ which is the closest to $x_0$. We consider the $n$ values $x_i:=x^\prime+\frac{iL}{n}$ ($i=1,\dots,n$). Then, using Lemma $1$ of \cite{mnh:ps}, there exists a point $\xi$ such that
$$
\frac{h^{(n)}(\xi)}{n!}=\sum_{i=0}^{n}\frac{h(x_i)}{\prod_{j \neq i}(x_i-x_j)}.
$$
Hence, using the inequality $|x_0-x_j| \ge \frac{jK}{n}$, we have
\begin{eqnarray*}
\left|\frac{h^{(n)}(\xi)}{n!}\right| & \ge & \frac{\sigma}{\prod_{j \neq 0}|x_0-x_j|}-\Delta\sum_{i=1}^{n}\frac{1}{\prod_{j \neq i}|x_i-x_j|}\\
& \ge & \frac{\sigma}{K^n}-\frac{\Delta n^n}{KL^{n-1}}\sum_{i=1}^{n}\frac{1}{\prod_{j \neq i}|i-j|}\\
& \ge & \frac{\sigma}{K^n}-\frac{\Delta (2n)^n}{n!KL^{n-1}}.
\end{eqnarray*}
From this, we deduce that we have
$$
\left|\frac{h^{(n)}(\xi)}{n!}\right| \ge \frac{\sigma}{2K^n}\quad\mbox{or}\quad\frac{\Delta (2n)^n}{n!KL^{n-1}} \ge \frac{\sigma}{2K^n},
$$
implying that \eqref{rep-arc} follows. Inequality \eqref{rep-arc1} then follows from
$$
|Kh^{(1)}(\xi)| \ge |(x^\prime-x_0)h^{(1)}(\xi)|=|h(x^\prime)-h(x_0)| \ge \sigma-\Delta,
$$
where we used the mean value theorem.
\end{proof}

\section{Proof of Theorem 1.1}

The proof provided here is very similar to that of the main theorem of \cite{mnh:ps}.

The result is trivial if $\delta=0$ or if $\Gamma_{\delta} \cap \Z^2$ is empty.  Thus we can assume that $\delta > 0$ and that $\S > 0$. Now, without any loss in generality, we can assume that $\alpha_n$ is irrational. Indeed, if $\alpha_n$ is rational, we choose an irrational number $\alpha'_n$ that satisfies
$$
|\alpha'_n-\alpha_n| < \min\left(\rho,\frac{\delta}{(2X)^n}\right),
$$
for some $\rho>0$, and replace $\delta$ by $\delta':=2\delta$. An upper bound for the number of solutions to this new system is also an upper bound for $\S$. We therefore continue with the initial notation here. We will use it again in the proof of Corollary 1.1.

We start by imposing a structure on the set of points of $\Gamma_{\delta} \cap \Z^2$. We consider the points ordered according to their first coordinate. We take the first $n+2$ points $M_1, \dots, M_{n+2}$ and evaluate $\Lambda(M_1, \dots, M_{n+2})$. If it is not zero then we start over with $M_{n+2}$ as the first point. If it is zero then these $n+2$ points are on a curve $\gamma$ of equation $y=P(x)$  with $deg\ P \le n$. We then consider the maximal value of $j \ge n+2$ for which each $\Lambda(M_1,\dots,M_{n+1},M_{i})=0$ for each $i=n+2, \dots, j$. Since the points $M_1,\dots,M_{n+1}$ are sufficient to define $\gamma$, we deduce that all the points $M_1,\dots,M_j$ are on $\gamma$, but not $M_{j+1}$. We then continue with $M_j$ as the first point. This process is repeated until less than $n+2$ points remain to construct $\Lambda$.

We have thus defined a sequence of set of points, $\A_1,\dots,\A_J$, for which each one either has that {\bf (i)} of Lemma \ref{minor} holds or is a major arc. Our plan is to use the first case to estimate the total contribution of the major arcs that have only a few points. Indeed, assume that $\A$ is composed of the points $M_1,\dots, M_j$. By construction, we have $|\Lambda(M_1,\dots,M_{n+1},M_{j+1})| \ge 1$ from which it follows that $|x_{j+1}-x_1| \ge \left(\frac{1}{(n+2)\delta}\right)^{\frac{2}{n(n+1)}}$. We deduce that the total contribution of all the sets that have at most $V:=n^2+n$ (for $n \ge 2$) points, and the remaining at most $n+1$ points, is bounded by
$$
\le 2V((n+2)\delta)^{\frac{2}{n(n+1)}}X+V.
$$

It remains to estimate the total contribution of the major arcs with more than $V$ points. We \lq\lq reset the notation" and assume that the remaining sequence of major arcs is $\A_1,\dots,\A_J$. Also for each such major arc $\A_i$ we denote by $Q_i$ the denominator and by $L_i$ the length. It is convenient to differentiate two cases here. The first case is the contribution of all the major arcs for which the denominator is $\ge \frac{c_2}{\delta}$ for some sufficiently small positive constant $c_2(=c_2(n))$. In this case, we see that the argument in the proof of Lemma \ref{qdensity} tells us that $L_i \ge Q_i^{\frac{1}{n}-\frac{n-1}{n(V-1)}}$ so that the contribution satisfies $\#\A_i \ll \frac{L_i}{Q_i^{\frac{1}{n}-\frac{n-1}{n(V-1)}}}$ and the total contribution is therefore
$$
\ll \delta^{\frac{1}{n}-\frac{n-1}{n(V-1)}}\sum_{i=1}^{J} L_i \ll \delta^{\frac{1}{n}-\frac{n-1}{n(V-1)}}X \ll \delta^{\frac{2}{n(n+1)}}X.
$$
For $n=1$ it is simpler and we take for example $V=3$ with Lemma \ref{qsol} instead to get the corresponding results.

We now consider the second case, that is when all the major arcs are of denominator less than $\frac{c_2}{\delta}$. By looking only at this subsequence we may have consecutive major arcs that have the same equation. But by construction, after each (but at most one) major arc $\A$, we have a point $M=(x_0,y_0)$ that is not a point satisfying the equation of $\A$. This point $M$ thus satisfies
\begin{eqnarray*}
|f(x_0)-P(x_0)| & = & |f(x_0)-y_0 -P(x_0)+y_0|\\
& \ge & |P(x_0)-y_0|-|f(x_0)-y_0|\ge\frac{1}{q}-\delta > \frac{1}{2q} > \delta.
\end{eqnarray*}
We deduce from Lemma \ref{mult} that there are at most $n$ consecutive such major arcs with the same equation. Furthermore, we can extract the proper major arcs with the largest contribution. We thus \lq\lq reset the notation" one last time and consider the sequence $\A_1,\dots,\A_J$ of such proper major arcs. In the end, their total contribution will be multiplied by $n$. Again their respective denominators are noted $Q_i$ and their lengths $L_i$. Also, following Lemma \ref{rep-loc}, the distance between the first point in $\A_{i}$ and the closest point in $\A_{i+1}$ that is not on the equation of $\A_i$ is denoted by $K_i$ ($i=1,\dots,J-1$). All the conditions are met to apply Lemma \ref{rep-loc} to the function $\phi(x):=f(x)-P(x)$ to the interval in the variable $x$ that contains $\A_i$ and the first point in $\A_{i+1}$. For $n \ge 2$, there are two cases to be considered. In the first one, we have
$$
K_i \gg \frac{L_i}{(Q_i\delta)^{1/(n-1)}}.
$$
In the second one, we use Lemma \ref{lon} with the function $\phi$ to get
$$
\frac{1}{\left|\alpha_n-\frac{a_n}{q_n}\right|^{1/n}} \gg \frac{L_i}{\delta^{1/n}},
$$
so that the second case satisfies
$$
K_i \gg \frac{1}{\left(\left|\alpha_n-\frac{a_n}{q_n}\right|Q_i\right)^{1/n}} \gg \frac{L_i}{(Q_i\delta)^{1/n}}.
$$
For $n=1$ there is only one case to consider. In short, we deduce that the inequality
$$
K_i \gg \frac{L_i}{(Q_i\delta)^{1/n}}
$$
holds for all $n \ge 1$. For $n \ge 2$, since the number of points is at least $v:=n+2$, it follows from Lemma \ref{qdensity} that the relation $\#\A_i \ll \frac{L_i}{Q_i^{\frac{1}{n}-\frac{n-1}{n(v-1)}}}$ holds. Now, since we have
$$
\sum_{i=1}^{J-1} K_i \ll X,
$$
we deduce that the total contribution of this case is
$$
\sum_{i=1}^{J-1} \#\A_i \ll \sum_{i=1}^{J-1} \frac{L_i}{Q_i^{1/n-\frac{n-1}{n(v-1)}}} \ll \sum_{i=1}^{J-1} \frac{K_i(Q_i\delta)^{1/n}}{Q_i^{1/n-\frac{n-1}{n(v-1)}}} \ll \delta^{\frac{2}{n(n+1)}}X.
$$
For $n=1$, Lemma \ref{qdensity} can by replaced by Lemma \ref{qsol} and the proof is then similar. We are thus left with at most one proper major arc and this one can absorb a contribution of size $\ll_{n} 1$ by increasing the multiplicative constant. The proof is thus complete.

\section{Proof of Corollary 1.1}

We apply Theorem 1.1. We clearly only need to evaluate $\mathcal{R}$. We assume that the polynomial \eqref{P} realizes the maximum. We will assume that $\delta > 0$ since otherwise we can take $\delta:= \frac{1}{2}\min_{\substack{x \in [X,2X] \cap \Z \\ \| P(x)\|>0}} \| P(x)\|$. We assume also that the maximum is larger than $n^2+n$ since otherwise the result is clear. It will be convenient to have an upper bound for the denominator of such a polynomial. Since $P(x)$ is of degree $d$, the equation $y=P(x)$ is completely determined by $d+1$ points $(x_1,y_1),\dots,(x_{d+1},y_{d+1})$. We find it by expanding
$$
0=\begin{vmatrix} 1 & x & x^2 & \cdots & x^d & y\\ 1 & x_1 & x_1^2 & \cdots & x_1^d & y_1\\ \vdots & \vdots & \vdots & \cdots & \vdots & \vdots \\ 1 & x_{d+1} & x_{d+1}^2 & \cdots & x_{d+1}^d & y_{d+1} \end{vmatrix}
$$
and in particular we retrieve the denominator $q$ as a divisor of the coefficient of $y$. This coefficient is a Vandermonde determinant so that it is bounded by $X^{\binom{n+1}{2}}$ since $d \le n$. This is the desired upper bound. We deduce from the well known inequality $\omega(q) \ll \frac{\log q}{\log \log q}$ ($q \ge 3$) that, for each fixed $n \ge 1$ and $\epsilon > 0$, we have $n^{\omega(q)} \ll_{n,\epsilon} X^\epsilon$.

We will examine three cases separately:
\begin{eqnarray}
& q \ge \frac{c_3}{\delta},\\
& q < \frac{c_3}{\delta} \quad \mbox{and} \quad |f(x)-P(x)| \le \frac{1}{3(2e)^nq} \quad \mbox{for each}\ x \in [X,2X],\\
& q < \frac{c_3}{\delta} \quad \mbox{and} \quad |f(z)-P(z)| > \frac{1}{3(2e)^nq} \quad \mbox{for a}\ z \in [X,2X],
\end{eqnarray}
where $c_3(=c_3(n))$ is a sufficiently small positive constant.

In the first case, the result follows from Lemma \ref{qsol} since we then find a contribution
$$
\ll \delta^{1/n}X+X^\epsilon.
$$

In the third case, since there are at least  $n^2+n+1$ solutions, we can choose (by Lemma \ref{mult}) the proper major arc $\A$ of length $L$ with the largest contribution. Then, the function $\phi(x):=f(x)-P(x)$ varies of at least $\gg 1/q$ between $z$ and the furthest $x$ such that $ (x,P(x)) \in \A$. By using Lemma \ref{rep-loc} as in the previous demonstration, we obtain
$$
K := |z-x| \gg \frac{L}{(q\delta)^{1/n}}.
$$
By Lemma \ref{qsol}, we obtain a contribution of at most
$$
\ll \frac{L}{q^{1/n}}+X^\epsilon \ll \delta^{1/n}K + X^\epsilon \ll \delta^{1/n}X + X^\epsilon.
$$

In the second case, using the notation introduced in the proof of Theorem 1.1 and then Lemma \ref{lon}, we obtain
$$
X \le |I| < \left(\frac{1}{3q\left|\alpha'_n-\frac{a_n}{q_n}\right|}\right)^{1/n},
$$
so that we have
$$
\left|\alpha'_n-\frac{a_n}{q_n}\right| < \frac{1}{3qX^n}\quad \mbox{and therefore}\quad \left|\alpha_n-\frac{a_n}{q_n}\right| < \frac{1}{2qX^n},
$$
if $\rho$ is taken small enough.

Now, let $\frac{r}{s}$ be any fraction with all the requested properties. If $s \le 2q_n$, the result follows from Lemma \ref{qsol}. We can therefore assume that $s > 2q_n$, in which case
$$
\frac{1}{sq_n} \le \left|\frac{r}{s}-\frac{a_n}{q_n}\right| \le \left|\frac{r}{s}-\alpha_n\right|+\left|\alpha_n-\frac{a_n}{q_n}\right| < \frac{1}{s^2} + \frac{1}{2qX^n}.
$$
From this we deduce that $s > X^n$ which is the desired contradiction and the result follows.

\section{Proof of Corollary 1.2}

The hypothesis \eqref{hypc2} implies that $\alpha_n$ is irrational, so we do not change its value. We apply Theorem 1.1. We only need to evaluate $\mathcal{R}$. We assume that the polynomial \eqref{P} realizes the maximum. By Lemma \ref{qsol} and the preceding proof we have the inequality
$$
\mathcal{R} \ll \frac{X}{q^{1/n}} +X^{\epsilon}.
$$
The result is clear if $q \gg \frac{1}{\delta^{2/(n+1)}}$. We may thus assume that $q_n \le q \ll \frac{1}{\delta^{2/(n+1)}}$. We set
$$
\left|\alpha_n-\frac{a_n}{q_n}\right| = \frac{1}{q^{\theta}_n}.
$$
By Lemma \ref{lon} we can write
$$
 L \ll (\delta q_n^{\theta})^{1/n}
$$
and so by using Lemma \ref{qsol} we find that the contribution is
$$
\ll \frac{L}{q_n^{1/n}} + X^{\epsilon} \ll (\delta q_n^{\theta-1})^{1/n}+X^\epsilon.
$$
The conclusion follows from the fact that $\delta\left( \frac{1}{\delta^{2/(n+1)}}\right)^{\frac{n+3}{2}-1} \ll 1$.

\vspace{1.5cm}

\noindent\textbf{Patrick Letendre}\\
D\'ep. math\'ematiques et statistique\\
Universit\'e Laval\\
Qu\'ebec\\
Qu\'ebec G1V 0A6\\
Canada\\
{\tt Patrick.Letendre.1@ulaval.ca}\\

\end{document}